\documentclass[reqno,b5paper]{amsart}
\usepackage{amsmath,amscd,pdfsync}
\usepackage{amssymb,stmaryrd}
\usepackage{amsthm}
\usepackage{enumerate}
\usepackage[mathscr]{eucal}
\theoremstyle{plain}
\newtheorem{theorem}{Theorem}[section]
\newtheorem{prop}[theorem]{Proposition}
\newtheorem{lemma}[theorem]{Lemma}
\newtheorem{corol}[theorem]{Corollary}

\theoremstyle{definition}

\newtheorem{remark}{\textnormal{\textbf{Remark}}}
\newtheorem*{acknowledgement}{\textnormal{\textbf{Acknowledgement}}}

\theoremstyle{remark}
\newtheorem{example}{Example}

\numberwithin{equation}{section}

\usepackage{hyperref}
\newcommand{\SLn}[1][n]{\mathrm{SL}(#1)}
\newcommand{\sln}[1][n]{\mathfrak{sl}(#1)}
\newcommand{\GLn}[1][n]{\mathrm{GL}(#1)}
\newcommand{\Ad}{\mathrm{Ad}}
\newcommand{\C}{\mathbf{C}}
\newcommand{\Z}{\mathbf{Z}}
\newcommand{\co}{\colon}
\newcommand{\rep}{r}
\usepackage{color}

\begin{document}
\title[Deformations of reducible representations]{Deformations of reducible representations of knot groups into $\mathrm{SL}(n,\mathbf{C})$}
\author[M. Heusener \and O. Medjerab]{Michael Heusener \and Ouardia Medjerab}
\newcommand{\acr}{\newline\indent}
\address{\llap{*\,}Clermont Universit\'e, Universit\'e Blaise Pascal,\acr 
Laboratoire de Mathématiques, BP 10448, F-63000 Clermont-Ferrand. \acr
CNRS, UMR 6620, LM, F-63171 Aubi\`ere, 
FRANCE}
\email{heusener@math.univ-bpclermont.fr}
\address{\llap{**\,}Ecole Normale Sup\'erieure de Kouba, Alger\acr
B.P.~92 Kouba 16050 \acr
ALGERIE}
\email{medjerab@ens-kouba.dz}
\begin{abstract}
Let $K$  be a knot in  $S^3$ and $X$ its complement. We study deformations of non-abelian,  metabelian, reducible 
 representations of the knot group  $\pi_1(X)$ 
into $\mathrm{SL}(n,\mathbf{C})$ which are associated to  a simple root of the  Alexander polynomial.
We prove that some of these metabelian  reducible  representations are smooth points of the  
$\mathrm{SL}(n,\mathbf{C})$-representation variety and that they have irreducible deformations.
\end{abstract}

\maketitle
\today

\bigskip

\section{Introduction} 
 Let $K$ be a knot  in  $S^3$ and $X=\overline{S^3\setminus V(K)}$  its complement, where $V(K)$ is
a tubular neighborhood of $K$. Moreover, let   $\Gamma_K=\pi_1(X)$ denote  the fundamental group of $X$. The aim of this paper is to study deformations of reducible metabelian representations  of 
$\Gamma_K$ into $\SLn[n,\C]$. The metabelian representations in question where introduced by G.~Burde \cite{Burde1967} and G.~de~Rham \cite{deRham1967}. 
Let us recall this result: for each nonzero complex number
$\lambda\in\C^*$ there exists a diagonal representation $\rho_\lambda\co\Gamma_K\to\SLn[2,\C]$ given by 
\[
   \rho_\lambda(\gamma)=
     \begin{pmatrix}
\lambda^{\varphi(\gamma)}& 0 \\
0 & \lambda^{-\varphi(\gamma)} 
\end{pmatrix}\,.
\]
Here $\varphi\co\pi_1(X)\to\Z$ denotes the canonical surjection which maps the meridian $\mu$ of 
$K$ to $1$ i.e.\ $\varphi(\gamma)=\mathrm{lk}(\gamma,K)$.
Burde and de~Rham proved that there exists a metabelian, non-abelian, reducible representation
\begin{equation}\label{eq:redmetab}
    \rho_\lambda^z \co \Gamma_K  \to  \SLn[2,\C], \quad 
   \rho_\lambda^z(\gamma)=
   \begin{pmatrix}
1 & z(\gamma) \\
0 & 1 \\
\end{pmatrix} 
     \begin{pmatrix}
\lambda^{\varphi(\gamma)}& 0 \\
0 & \lambda^{-\varphi(\gamma)} 
\end{pmatrix}
\end{equation}
if and only if $\lambda^2$ is a root of the Alexander polynomial $\Delta_K(t)$.
Recall that a representation $\rho\co G\to \GLn[n,\C]$ of a group $G$ is called reducible if
the image $\rho(G)$ preserves a proper subspace of $\C^n$.
Otherwise, $\rho$ is called \emph{irreducible}.

The question whether or not the representation $\rho_\lambda^z$ is a limit of irreducible representations of $\Gamma_K$ into $\SLn[2,\C]$ was studied in \cite{Heusener-Porti-Suarez2001}.
Theorem~1.1 of \cite{Heusener-Porti-Suarez2001} states that a metabelian, non-abelian, reducible representation
$\rho_\lambda^z\co\Gamma_K\to\SLn[2,\C]$ 
is the limit of irreducible representations if 
$\lambda^2$ is a simple root of $\Delta_K(t)$. Moreover, in this case the representation 
$\rho_\lambda^z$ is a smooth point of the representation variety $R(\Gamma_K,\SLn[2,\C])$; 
it is contained in a unique $4$-dimensional component $R_\lambda\subset R(\Gamma_K,\SLn[2,\C])$.

This article studies the behavior of the representations in question under the composition with
the $n$-dimensional, irreducible, rational representation 
$\rep_n\co\SLn[2,\C]\to\SLn[n,\C]$. 
(for more details see Section~\ref{sec:sl2}).
It is proved in Proposition~\ref{prop:irred-exist} that generically for an irreducible representation
$\rho\in R_\lambda$ the representation $\rho_n:=\rep_n\circ\rho\in R(\Gamma_K,\SLn[n,\C])$ is also irreducible.
The main result of this article is the following:
 \begin{theorem}\label{thm:smoothR_n}
 If $\lambda^2$ is a simple root of $\Delta_K(t)$ and if
 $\Delta_K(\lambda^{2k})\ne0$ for $2\leq k \leq n-1$ then the reducible metabelian representation
$\rho_{\lambda,n}^z:=\rep_n\circ \rho_\lambda^z$ is a limit of irreducible representations.
More precisely, $\rho_{\lambda,n}^z$ is a smooth point of $R(\Gamma_K,\SLn[n,\C])$; it is contained in a unique 
$(n+2)(n-1)$-dimensional component $R_{\lambda,n}\subset R(\Gamma_K,\SLn[n,\C])$.
 \end{theorem}

\begin{remark}
Let $\rho_{\lambda,n}\co\Gamma_K \to \SLn[n,\C]$ be the diagonal representation 
given by $\rho_{\lambda,n}=\rep_n\circ\rho_\lambda$. The group $\SLn[n,\C]$ acts on the 
representation variety $R(\Gamma_K,\SLn[n,\C])$ by conjugation, and
the orbit $\mathcal O (\rho_{\lambda,n})$ is contained in the closure 
$\overline{\mathcal O (\rho_{\lambda,n}^z)}$. Hence $\rho_{\lambda,n}$ and $\rho_{\lambda,n}^z$ 
project to the same point $\chi_{\lambda,n}$ of the \emph{character variety}
$$
X(\Gamma_K,\SLn[n,\C])= R(\Gamma_K,\SLn[n,\C])\mathop{\sslash} \SLn[n,\C]\,.
$$
Here $R(\Gamma_K,\SLn[n,\C])\mathop{\sslash} \SLn[n,\C]$ denotes the GIT quotient of the action (see \cite{Newstead1978} for more details). Recall that the GIT quotient parametrizes the closed orbits of the $\SLn[n,\C]$  action.

It is possible to study the local picture of the character variety at $\chi_{\lambda,n}$ as done in 
\cite{Heusener-Porti-Suarez2001} and
\cite{Heusener-Porti2005}. 
Unfortunately, there are additional technical difficulties, 
and the computations necessary are much more involved. 
These complications are due to the fact that the diagonal representation $\rho_{\lambda,n}$ is contained in
$2^{n-1}$ components of $R(\Gamma_K,\SLn[n,\C])$. Nevertheless, only the component
$R_{\lambda,n}$ contains irreducible representations.
We will address this topic in a forthcoming paper.
\end{remark}

P.~Menal-Ferrer and J.~Porti 
\cite{MenalFerrer-Porti2012}  showed that the conclusions of the above theorem hold for hyperbolic knots 
 if $\rho_\lambda^z$ is replaced by a lift  of the holonomy, 
 $\widetilde{\mathrm{hol}}\co\pi_1(S^3\smallsetminus K)\to\SLn[2,\C]$,  
 of the hyperbolic structure of the complement $S^3\smallsetminus K$. Note that Theorem~\ref{thm:smoothR_n} and Proposition~\ref{prop:irred-exist} do apply to non-hyperbolic knots. Irreducible metabelian representations and their deformations are studied by H.~Boden and S.~Friedl in a series of articles \cite{Boden-Friedl2008,Boden-Friedl2011,Boden-Friedl2013,Boden-Friedl2014}. 
 In particular the deformations of \emph{irreducible} metabelian representations,
 which are not considered in this paper, are studied in
 \cite{Boden-Friedl2013}.

 This article is organized as follows: in Section~\ref{sec:Notation} we will introduce notation and recall some facts which are used. In Section~\ref{sec:deform} we will prove that the representation variety
$R(\Gamma_K,\SLn[n,\C])$ contains an irreducible representation if $\Delta_K(t)$ has a simple root (see Proposition~\ref{prop:irred-exist}). Moreover we give a streamlined proof of a slightly generalized version of the deformation result used in \cite{MenalFerrer-Porti2012,BenAbdelghani-Heusener-Jebali2010,Boden-Friedl2013}
 (see Proposition~\ref{prop:smoothpoint}).
The necessary cohomological calculations and the basic facts about the representation theory of 
$\SLn[2,\C]$ are presented in Section~\ref{sec:cohocalculation},
in order to prove our main result, Theorem~\ref{thm:smoothR_n}.
Finally, in Section~\ref{sec:examples} some examples are exhibited.

\begin{acknowledgement} The authors thank Micha\"el Bulois and Simon Riche for pointing out the references \cite{Richardson1979} and \cite{Popov2008}. 
They also thank the anonymous referee for helpful comments and hints.

The first author was supported by the ANR project SGT and
ANR project ModGroup. The second author acknowledges support from the
Algerian \emph{Minist\`ere de l'Enseignement Sup\'erieur et de la Recherche Scientifique.}
\end{acknowledgement}

\section{Notation and Facts}\label{sec:Notation}
 To shorten notation we write $\SLn$, $\GLn$ and $\sln$ instead of $\SLn[n,\C]$, $\GLn[n,\C]$ and $\sln[n,\C]$.
 
Let $\varphi\co\pi_1(X)\to\Z$ denote the canonical surjection which maps the meridian $\mu$ of $K$ to $1$
i.e.\ $\varphi(\gamma)=\mathrm{lk}(\gamma,K)$. 
We associate to a nonzero complex number $\alpha \in \C^{\ast}$ a homomorphism
\[
     \alpha^\varphi\co\Gamma_K \to \C^{*} ,\quad
  \alpha^\varphi\co\gamma  \mapsto   \alpha^{\varphi(\gamma)}\,.
\]
Note that $\alpha^\varphi$ maps the meridian $\mu$ of $K$ to $\alpha$. 
We define $\C_{\alpha}$ to be the $\Gamma_K$-module $\C$ with the action induced by 
$\alpha^\varphi$, i.e. 
$\gamma\cdot x=\alpha^{\varphi(\gamma)}x$ for all $\gamma \in \Gamma_K$ and all $x \in \C$.
The trivial $\Gamma_K$-module $\C_1$ is simply denoted $\C$.
With this notation it is easy to see that a map 
\[
    \rho_\lambda^z \co \Gamma_K  \to  \SLn[2,\C], \quad 
   \rho_\lambda^z(\gamma)=
   \begin{pmatrix}
1 & z(\gamma) \\
0 & 1 \\
\end{pmatrix} 
     \begin{pmatrix}
\lambda^{\varphi(\gamma)}& 0 \\
0 & \lambda^{-\varphi(\gamma)} 
\end{pmatrix}
\]
 is a homomorphism if and only if the map
 $z\co\Gamma_K\to\C_{\lambda^2}$ is a $1$-cocycle i.e.\
 $z(\gamma_1\gamma_2)= z(\gamma_1) + \lambda^{2\varphi(\gamma_1)}z(\gamma_2)$. 
Note also that $\rho_\lambda^z$ is abelian if $\lambda=\pm1$. If $\lambda^2\neq1$ then 
$\rho_\lambda^z$ is abelian if and only if $z$ is a coboundary i.e.\ there exists an element 
$x_0\in\C$ such that $z(\gamma) = (\lambda^{2\varphi(\gamma)} -1 )\,x_0$.
The general reference for group cohomology is Brown's book \cite{Brown1982}.

In what follows we are mainly interested in the following situation:
let $X$ be the complement of a knot $K\subset S^3$ and let 
$A$ be a $\pi_1(X)$-module.
The spaces $X$ and $\partial X$ are aspherical and hence the natural
homomorphisms $H^*(\pi_1(X) ; A)\to H^*(X ; A)$ and $H^*(\pi_1(\partial X) ; A)\to H^*(\partial X ; A)$ are isomorphisms.  Moreover, the knot complement $X$ has the homotopy type of a $2$-dimensional CW-complex which implies that 
$H^k(\pi_1(X) ; A)=0$ and $H_k(\pi_1(X) ; A)=0$ for $k\geq3$. 
See \cite{Whitehead1978} for more details.

The Laurent polynomial ring $\C[t^{\pm1}]$ 
turns into a $\Gamma_K$-module via the action $\gamma\, p(t) = t^{\varphi(\gamma)}\,p(t)$ for all $\gamma \in \Gamma_K$ and all 
$p(t) \in \C[t^{\pm1}]$. Recall that there are isomorphisms of $\C[t^{\pm1}]$-modules
\[
H_*(\Gamma_K; \C[t^{\pm1}]) \cong H_*(X; \C[t^{\pm1}]) \cong
H_*(X_\infty; \C)
\]
where $X_\infty$ denotes the infinite cyclic covering of the knot complement $X$ (see \cite[Chapter~5]{Davis-Kirk2001}). 
The module $H_1(\Gamma_K; \C[t^{\pm1}])$ is a finitely generated torsion module  called the \emph{Alexander module} of $K$. A generator of its order ideal is called the \emph{Alexander polynomial} $\Delta_{K}(t) \in \C [t^{\pm 1}]$ of $K$. The Alexander polynomial is unique up to multiplication with a unit in $\C[t^{\pm1}]$.

For completeness we will state the next lemma which shows that
the cohomology groups $H^*(\Gamma_K;\C_\alpha)$ are determined by the Alexander module
$H_1(\Gamma_K; \C[t^{\pm1}])$. 
\begin{lemma}\label{lem:H1GammaC} Let $K\subset S^3$ be a knot and $\Gamma_K$ its group.
Let $\alpha\in\C^*$ be a nonzero complex number and let $\C_\alpha$ denote the $\Gamma_K$-module given by the action $\gamma\, z = \alpha^{\varphi(\gamma)} z$. 

If $\alpha=1$ then $\C_\alpha=\C$ is a trivial $\Gamma_K$-module and
$H^k(\Gamma_K, \C)\cong \C$ for $k=0,1$ and 
$H^k(\Gamma_K, \C)=0$ for $k\geq 2$.
If $\alpha\neq1$ then $H^0(\Gamma_K, \C_\alpha)= 0$  and 
$\dim H^1(\Gamma_K, \C_\alpha)= \dim H^2(\Gamma_K, \C_\alpha)$. Moreover,
$H^1(\Gamma_K, \C_\alpha)\neq 0 $ if and only if $\Delta_K(\alpha)=0$.
\end{lemma}
\begin{proof} 
We have $H_0(X_\infty; \C) \cong \C \cong \C[t^{\pm1}]/ (t-1)$ and
$H_k(X_\infty; \C)=0$ for $k\geq 2$ (see \cite[Prop.~8.16]{BZH2013}). 
If $\alpha=1$ then $H^k(\Gamma_K, \C)\cong \C$ for $k=0,1$ and 
$H^k(\Gamma_K, \C)=0$ for $k\geq 2$ follows.

Now suppose that $\alpha\in\C^*$, $\alpha\neq1$, and notice that we have an isomorphism $\C_\alpha\cong\C[t^{\pm1}]/ (t-\alpha)$. The cohomology group $H^0(\Gamma_K, \C_\alpha)$ vanishes, since
the $\Gamma_K$-module $\C_\alpha$ has no invariants, and $H^k(\Gamma_K, \C_\alpha)=0$ for 
$k> 2$ since the knot complement $X$ has the homotopy type of a $2$-complex.
Recall that the Alexander module $H_1(\Gamma_K ; \C[t^{\pm1}])$ is finitely generated torsion module and hence a sum of non-free cyclic modules since $\C[t^{\pm1}]$ is a principal ideal domain. The Alexander polynomial is the order ideal of 
$H_1(\Gamma_K ; \C[t^{\pm1}])$. 
Since $\alpha\neq1$, it follows from the universal coefficient theorem that
$H^1(\Gamma ; \C_\alpha)\cong \mathrm{Hom}(H_1(\Gamma_K; \C); \C_\alpha)$.
Hence $H^1(\Gamma_K, \C_\alpha)\neq 0 $ if and only if the module $H_1(\Gamma_K; \C)$
has $(t-\alpha)$-torsion which is equivalent to $\Delta_K(\alpha)=0$.
Finally, $\dim H^1(\Gamma_K, \C_\alpha)= \dim H^2(\Gamma_K, \C_\alpha)$ follows since the Euler characteristic of $X$ vanishes. See also \cite[Proposition 2.1]{BenAbdelghani2000} for more details.
\end{proof}

In what follows we will also make use of the Poincar\'e-Lefschetz duality theorem with twisted coefficients:
let $M^m$ be a connected, orientable, compact $m$-dimensional manifold with boundary $\partial M^m$ and let
$\rho\co\pi_{1}(M^m)\to \SLn$ be a representation. Then the cup-product and the Killing form $b\co\sln_\rho\otimes\sln_\rho\to\C$ 
 induce a non-degenerate bilinear pairing
\begin{multline}\label{eq:poincare} H^k(M^m ; \sln_\rho)\otimes H^{m-k}(M^m,\partial M^m ; \sln_\rho)\xrightarrow{\smallsmile}\\
H^{m}(M^m,\partial M^m ; \sln_\rho\otimes\sln_\rho)\xrightarrow{b}H^{m}(M^m,\partial M^m ; \C)\cong \C
\end{multline}
and hence an isomorphism $H^k(M^m ; \sln_\rho)\cong H^{m-k}(M^m,\partial M^m ; \sln_\rho)^*$, for all $0\leq k\leq m$.
See \cite{Johnson-Millson1987,Porti1995} for more details.

 \subsection{Group cohomology and representation varieties}
\label{sec:cohorep}
%

Let now $\Gamma$ be a finitely generated group.
The set $R_n(\Gamma):=R(\Gamma,\SLn)$ of homomorphisms of $\Gamma$ in 
$\SLn$ is called the $\SLn$-repre\-sentation variety of $\Gamma$ and has the structure of a (not necessarily irreducible) algebraic set.

Let $\rho\co\Gamma\to \SLn$ be a representation. 
The Lie algebra $\sln$ turns into a $\Gamma$-module via 
$\Ad\,\rho$. This module will be simply denoted by 
$\sln_\rho$. A cocycle $d\in Z^1(\Gamma;\sln_\rho)$ is a map 
$d\co\Gamma\to \sln$ satisfying 
\[d(\gamma_1\gamma_2)=d(\gamma_1)+\rho(\gamma_1)\,d(\gamma_2)\,\rho(\gamma_1)^{-1}\quad,\ \forall\ \gamma_1,\ \gamma_2\in\Gamma\,.\]

It was observed by Andr\'e Weil \cite{Weil1964} that there is a natural inclusion of the Zariski tangent space $T_\rho^\mathit{Zar}(R_n(\Gamma))\hookrightarrow Z^1(\Gamma ; \sln_\rho)$. Informally speaking, given a smooth curve $\rho_\epsilon$ of representations through $\rho_0=\rho$ one gets a $1$-cocycle $d\co\Gamma\to \sln$ by defining
\[ d(\gamma) := \left.\frac{d \, \rho_{\epsilon}(\gamma)}
{d\,\epsilon}\right|_{\epsilon=0} \rho(\gamma)^{-1},
\quad\forall\gamma\in\Gamma\,.\]

It is easy to see that the tangent space to the orbit by conjugation corresponds to the space of $1$-coboundaries $B^1(\Gamma ; \sln_\rho)$. Here, $b\co\Gamma\to \sln$ 
is a coboundary if there exists 
$x\in \sln$ such that $b(\gamma)=\rho(\gamma)\,x\,\rho(\gamma)^{-1}-x$. A detailed account can be found in \cite{Lubotzky-Magid1985}.

Let $\dim_\rho R_n(\Gamma)$ be the local dimension of $R_n(\Gamma)$ at $\rho$ (i.e.\ the maximal dimension of the irreducible components of $R_n(\Gamma)$ containing $\rho$ 
\cite[Ch.~II]{Shafarevich1977}). So we obtain:
\[ \dim_\rho R_n(\Gamma) \leq \dim T_\rho^\mathit{Zar}(R_n(\Gamma))\leq 
\dim Z^1(\Gamma ; \sln_\rho)\,.\]
We will call a representation $\rho\in R_n(\Gamma)$  \emph{regular} if 
$\dim_\rho R_n(\Gamma) = \dim Z^1(\Gamma ; \sln_\rho)$.

The following lemma follows
(for more details see \cite[Lemma~2.6]{Heusener-Porti-Suarez2001}):
\begin{lemma}\label{lem:smoothness}
Let $\rho\in R_n(\Gamma)$ be a representation. If $\rho$ is regular, then $\rho$ is a smooth point of the representation variety $R_n(\Gamma)$ and $\rho$ is contained in a unique component of $R_n(\Gamma)$ of dimension 
$\dim Z^1(\Gamma;\sln_\rho)$.
\end{lemma}
Note that there are discrete groups and representations $\rho$ which are smooth
points of the representation variety without been regular. (See
\cite[Example~2.10]{Lubotzky-Magid1985} for more details.)

\section{Deforming representations}
\label{sec:deform}

The aim of the following sections is to prove that the representation $\rho_{\lambda,n}^z$ from the introduction is a  smooth point of the representation variety.
We present a more streamlined and slightly generalized version of the deformation result from 
\cite{Heusener-Porti-Suarez2001,Heusener-Porti2005,BenAbdelghani-Heusener-Jebali2010,MenalFerrer-Porti2012,Boden-Friedl2014} (see Proposition~\ref{prop:smoothpoint}).
For the convenience of the reader we recall the setup.

First we will prove that the representation $\rho_{\lambda,n}^z\in R_n(\Gamma_K)$ is the limit of irreducible representations if $\lambda^2$ is a simple root of of the Alexander polynomial $\Delta_K(t)$. In what follows a property of an irreducible algebraic variety $Y$ is said to be true \emph{generically} if it holds except on a proper Zariski-closed subset of $Y$, in other words, if it holds on a non-empty Zariski-open subset.

Let $K\subset S^3$ be a knot, $\lambda^2\in\C$ a simple root of $\Delta_K(t)$ and
$z\in Z^1(\Gamma_K,\C_{\lambda^2})$ a cocycle representing a generator of 
$H^1(\Gamma_K,\C_{\lambda^2})$. Following \cite[Thm~1.1]{Heusener-Porti-Suarez2001} the representation $\rho_\lambda^z\in R_2(\Gamma_K)$ is a smooth point of the representation variety.
It is contained in an unique irreducible $4$-dimensional component 
$R_\lambda\subset R_2(\Gamma_K)$. Note that generically a representation $\rho\in R_\lambda$ is irreducible.
\begin{prop}\label{prop:irred-exist}
Let $K\subset S^3$ be a knot, $\lambda^2\in\C$ a simple root of $\Delta_K(t)$ and let
$z\in Z^1(\Gamma_K,\C_{\lambda^2})$ be a cocycle representing a generator of 
$H^1(\Gamma_K,\C_{\lambda^2})$. 

Then the representation 
$\rho_{\lambda,n}^z=r_n\circ \rho_{\lambda}^z\co\Gamma_K\to B_n$ is the limit of irreducible representation in $R_n(\Gamma_K)$. More precisely, generically a representation 
$\rho_n = r_n\circ \rho$, $\rho\in R_\lambda$ is irreducible.
\end{prop}

\begin{proof}
It follows from \cite[Theorem 1.1]{Heusener-Porti-Suarez2001} that 
$\rho_{\lambda}^z\in R_2(\Gamma_K)$ is the limit of irreducible representations. 
Moreover, $\rho_{\lambda}^z\in R_2(\Gamma_K)$ is a smooth point which is contained in a unique $4$-dimensional component $R_{\lambda}\subset R_2(\Gamma_K)$.

Let $\Gamma$ be a discrete group and let $\rho\co\Gamma\to\SLn[2,\C]$ be an irreducible representation.
If the image $\rho(\Gamma)\subset\SLn[2,\C]$ is Zariski-dense then the representation $\rho_n := r_n\circ\rho\in R_n(\Gamma)$ is irreducible. Hence in order to prove the proposition we show that there is a neighborhood $U=U(\rho_\lambda^z)\subset R_2(\Gamma_K)$ such that
$\rho(\Gamma)\subset\SLn[2,\C]$ is Zariski-dense for each irreducible $\rho\in U$.
Let now $\rho\co\Gamma_K\to\SLn[2,\C]$ be any irreducible representation and let $G\subset\SLn[2]$ denote the Zariski-closure of $\rho(\Gamma_K)$. Suppose that $G\neq\SLn[2]$.
Since $\rho$ is irreducible it follows that $G$ is, up to conjugation, not a subgroup of upper-triangular matrices of $\SLn[2]$. 
Then by \cite[Sec.~1.4]{Kovacic1986} and \cite[Theorem~4.12]{Kaplansky1957}
there are, up to conjugation, only two cases left:
\begin{itemize}
\item $G$ is a subgroup of the infinite dihedral group
\[ D_\infty =\Big\{ \big(\begin{smallmatrix} \alpha & 0\\ 0 &\alpha^{-1}\end{smallmatrix}\big)
\,\big|\, \alpha \in\C^*\Big\} \cup
\Big\{ \big(\begin{smallmatrix} 0 & \alpha\\ -\alpha^{-1}&0\end{smallmatrix}\big)
\,\big|\, \alpha \in\C^*\Big\}\,.\]

\item $G$ is one of the groups $A_4^{\SLn[2]}$ (the tetrahedral group), 
$S_4^{\SLn[2]}$ (the octahedral group) or $A_5^{\SLn[2]}$ (the icosahedral group). 
These groups are the preimages  in $\SLn[2]$ of the subgroups 
$A_4$, $S_4$, $A_5 \subset \mathrm{PSL}(2,\C)$.
\end{itemize}

In the first case it follows directly from 
\cite{Nagasato2007} that if $\rho$ is an irreducible metabelian representation
then the trace of the image of a meridian 
$\mathrm{tr}(\rho(\mu))=0$ i.e. $\rho(\mu)$ is similar to 
$\pm\big(\begin{smallmatrix} i & 0 \\ 0 &-i \end{smallmatrix}\big)$. Now, $\mathrm{tr}(\rho_\lambda^z(\mu))\neq 0$ since
$\Delta_K(-1)\neq0$ and $\Delta_K(\lambda^{\pm2})=0$.
For the second case  there are up to conjugation only finitely many irreducible representations of $\Gamma_K$ onto the subgroups
$A_4^{\SLn[2]}$, $S_4^{\SLn[2]}$ and $A_5^{\SLn[2]}$. Note that these finitely many orbits are closed and $3$-dimensional. Hence the irreducible $\rho\in R_\lambda$ such that $r_n\circ\rho$ is reducible is contained in a Zariski-closed subset of $R_\lambda$. Hence generically $r_n\circ\rho$ is irreducible for $\rho\in R_\lambda$.
\end{proof}

\begin{remark}
Recall that a finite group has only finitely many irreducible representations (see \cite{Serre1978,Fulton-Harris1991}). Hence, the restriction of $r_n$ to the groups 
$A_4^{\SLn[2]}$, $S_4^{\SLn[2]}$ and $A_5^{\SLn[2]}$ is reducible, for all but finitely many 
$n\in\mathbf{N}$.
\end{remark}

In order to prove that a certain representation $\rho\in R_n(\Gamma)$ is a smooth point of the representation variety
we will prove that  every cocycle $u\in Z^1(\Gamma_K;\sln_{\rho})$ is integrable.
In order to do this, we use the classical approach, i.e.\ we first solve the corresponding formal problem and apply then a  theorem of Artin \cite{Artin1968}. 

The formal deformations of a representation $\rho\co\Gamma\to \SLn$ are in general determined by an infinite sequence of obstructions (see \cite{Goldman1984, BenAbdelghani2000,Heusener-Porti-Suarez2001}).
In what follows we let
$C^1(\Gamma;\sln):=\{ c\co\Gamma\to\sln\}$ denote the $1$-cochains of $\Gamma$ with coefficients in $\sln$ (see \cite[p.59]{Brown1982}).

Let $\rho\co\Gamma\to \SLn$ be a representation. A formal deformation of $\rho$ is a homomorphism 
$\rho_\infty\co\Gamma\to \SLn[n,\C\llbracket  t \rrbracket]$

\[\rho_\infty(\gamma)=\exp\left(\sum_{i=1}^{\infty}t^iu_i(\gamma)\right)\rho(\gamma)\, ,
\quad u_i\in C^1(\Gamma;\sln) \]
such that $\mathrm{ev}_0\circ\rho_\infty=\rho$. Here $\mathrm{ev}_0\co \SLn[n,\C\llbracket t\rrbracket]\to \SLn$ is the evaluation homomorphism at $t=0$ and 
$\C\llbracket t \rrbracket$ denotes the ring of formal power series.

 We will say that $\rho_\infty$ is a formal deformation up to  order $k$ of $\rho$ if $\rho_\infty$ is a homomorphism modulo $t^{k+1}$.

An easy calculation gives that $\rho_\infty$ is a homomorphism up to first order if and only if 
$u_1\in Z^1(\Gamma ; \sln_\rho)$ is a cocycle. We call a cocycle 
$u_1\in Z^1(\Gamma;\sln_\rho)$ \emph{integrable} if there is a formal deformation of $\rho$ with leading term $u_1$.

\begin{lemma}
Let $u_1,\ldots,u_k\in C^1(\Gamma;\sln_\rho)$ such that
\[\rho_k(\gamma)=\exp\left(\sum_{i=1}^k t^iu_i(\gamma)\right)\rho(\gamma)\]
is a homomorphism into $\SLn[n,\C\llbracket t\rrbracket/(t^{k+1})]$. 
Then there exists an obstruction class 
$\zeta_{k+1}:=\zeta_{k+1}^{(u_1,\ldots,u_k)}\in H^2(\Gamma, \sln_\rho)$ with the following properties:
\begin{enumerate}[(i)]
\item  There is a cochain $u_{k+1}\co\Gamma\to \sln_\rho$ such that
\[\rho_{k+1}(\gamma)=\exp\left(\displaystyle\sum_{i=1}^{k+1}t^iu_i(\gamma)\right)\rho(\gamma)\]
is a homomorphism modulo $t^{k+2}$ if and only if $\zeta_{k+1}=0$.
\item  The obstruction $\zeta_{k+1}$ is natural, i.e. if $f\co\Gamma_1\to\Gamma$ is a homomorphism then $f^*\rho_k:=\rho_k\circ f$ is also a homomorphism modulo $t^{k+1}$ and $f^*(\zeta_{k+1}^{(u_1,\ldots,u_k)})=\zeta_{k+1}^{(f^*u_1,\ldots,f^*u_k)}\in 
H^2(\Gamma_1;\sln_{f^*\rho})$.
\end{enumerate}
\end{lemma}

\begin{proof}
The proof is completely analogous to the proof of 
Proposition~3.1 in \cite{Heusener-Porti-Suarez2001}. 
We replace $\SLn[2]$ and $\sln[2]$ by 
$\SLn$ and $\sln$ respectively. 
\end{proof}

The following result streamlines the arguments given in \cite{Heusener-Porti2005} and
\cite{BenAbdelghani-Heusener-Jebali2010}:

\begin{prop}\label{prop:smoothpoint}
Let $M$ be a connected, compact, orientable $3$-manifold with torus boundary
and let
$\rho\co\pi_1M\to \SLn$ be a representation.

If $\dim H^1(\pi_1M; \sln_{\rho}) =n-1$ then $\rho$   is a smooth point of the 
$\SLn$-representation variety $R_n(\pi_1M)$. Moreover,
$\rho$ is contained in a unique component of dimension
$n^2+n-2 - \dim H^0(\pi_1M;\sln_\rho)$.
\end{prop}
\begin{proof}
First we will show that the map 
$i^*\co H^2(\pi_1 M;\sln_\rho)\to H^2(\pi_1 \partial M;\sln_{\rho})$ induced by the inclusion
$\partial M\hookrightarrow M$ is injective. 

Recall that for any CW-complex $X$ with
$\pi_1(X)\cong \pi_1(M)$ and for any $\pi_1 M$-module $A$ 
there are natural morphisms $H^i(\pi_1 M;A)\to H^i( X ;A)$ which are
isomorphisms
for $i=0,1$ and an injection for $i=2$ (see \cite[Lemma 3.3]{Heusener-Porti2005}).
Note also that $\partial M\cong S^1\times S^1$ is aspherical and hence
$H^*(\pi_1 \partial M;A)\to H^*( \partial M ;A)$ is an isomorphism.

First we will prove that for every representation $\varrho\in R_n(\Z\oplus\Z)$ we have
\begin{equation}\label{eq:boundary} 
\dim H^0(\Z\oplus\Z;\sln_\varrho) = \frac12 \dim H^1(\Z\oplus\Z;\sln_\varrho)\geq n-1\,.
\end{equation}
Moreover, we will prove  that $\varrho\in R_n(\Z\oplus\Z)$ is  regular if
and only if equality holds in~\eqref{eq:boundary}.
It follows from Poincar\'e duality \eqref{eq:poincare} that for every $\varrho\in R_n(\Z\oplus\Z)$ we have
\[ \dim H^0(\partial M; \sln_{\varrho})= \dim H^2(\partial M; \sln_{\varrho})\,, \]
and since the  Euler characteristic of $M$ vanishes we obtain the first equality in~\eqref{eq:boundary}:
\[
\dim H^1(\partial M; \sln_{\varrho}) = 2 \dim H^0(\partial M; \sln_{\varrho}) = 2 \dim \sln_\varrho^{\Z\oplus\Z}\,.
\]
Now, R.W.~Richardson proved in \cite[Thm.~C]{Richardson1979} that  the 
representation variety $R_n(\Z\oplus\Z)$ is an irreducible algebraic variety of dimension 
$(n+2)(n-1)$. 
Hence we obtain for every $\varrho\in R_n(\Z\oplus\Z)$ that
$$\dim Z^1(\partial M; \sln_{\varrho}) \geq (n+2)(n-1)=n^2+n-2$$
where the equality holds if and only if $\varrho$ is regular (see Lemma~\ref{lem:smoothness}).
At the same time, we have:
\begin{align*}
\dim Z^1(\partial M; \sln_{\varrho}) &= \dim H^1(\partial M; \sln_{\varrho}) + \dim B^1(\partial M; \sln_{\varrho})
\intertext{and the exactness of 
$0\to H^0(\partial M; \sln_{\varrho})\to\sln\to B^1(\partial M; \sln_{\varrho})\to 0$ gives}
\dim B^1(\partial M; \sln_{\varrho}) &= \dim \sln - \dim H^0(\partial M; \sln_{\varrho})\,.
\end{align*}
This together with $\dim H^1(\partial M; \sln_{\varrho}) = 2 \dim H^0(\partial M; \sln_{\varrho})$ gives
for all $\varrho\in R_n(\Z\oplus\Z)$:
\[
\dim Z^1(\partial M; \sln_{\varrho}) = \dim H^0(\partial M; \sln_{\varrho}) + n^2-1\geq n^2+n-2\,.
\]
It follows that  
\begin{equation}\label{eq:boundary2}
\dim H^0(\partial M; \sln_{\varrho})\geq n-1,\ \text{ for all $\varrho\in R_n(\Z\oplus\Z)$,} 
\end{equation}
and  $\varrho\in R_n(\Z\oplus\Z)$ is regular if and only if
$\dim H^0(\Z\oplus\Z;\sln_\varrho)=n-1$ (see also \cite{Popov2008}).

Now, the exact cohomology sequence of the pair $(M,\partial M)$ gives 
\begin{multline*}
\to H^1(M,\partial M; \sln_{\rho})\\ \to
H^1(M; \sln_{\rho})\xrightarrow{\alpha} H^1(\partial M; \sln_{\rho}) 
\xrightarrow{\beta} H^2(M,\partial M; \sln_{\rho})\\
\to H^2(M ; \sln_\rho)\xrightarrow{i^*}H^2(\partial M ; \sln_\rho)\to  
H^3(M,\partial M; \sln_{\rho})\to 0\,.
\end{multline*}
Poincar\'e-Lefschetz duality \eqref{eq:poincare} implies that $\alpha$ and $\beta$ are dual to each other.
This together with \eqref{eq:boundary2} gives:
\begin{multline*}
n-1=\dim H^1(M; \sln_{\rho})\geq \mathrm{rk} (\alpha) =\frac 1 2 
\dim H^1(\partial M; \sln_{\rho})\\= \dim H^0(\partial M; \sln_{\rho}) \geq n-1\,.
\end{multline*}
Therefore, $\dim H^0(\partial M;\sln_\rho)=n-1$ holds in Equation~\eqref{eq:boundary}, and consequently  
$i^*\rho=\rho\circ i_\# \in R_n(\partial M)$ is regular (here $i\colon\partial M\to M$ is the inclusion). 
Note also that $\beta$ is surjective, and hence
\[i^*\co H^2(M ; \sln_\rho)\to H^2(\partial M ; \sln_\rho)\] 
is injective. 
The following commutative diagram shows that 
$i^*\co H^2(\pi_1 M;\sln_\rho)\to H^2(\pi_1 \partial M;\sln_{\rho})$ is also injective:
$$
\begin{CD}
    H^2(M;\sln_{\rho}) @>{i^*}>> H^2(
\partial M;\sln_{\rho}) \\
    @AAA         @AA\cong A \\
    H^2(\pi_1 M;\sln_{\rho}) @>{i^*}>>
    H^2(\pi_1{\partial M};\sln_{\rho})\,.
\end{CD} $$

In order to prove that $\rho$ is a smooth point of $R_n(\pi_1M)$, we show that all cocycles in $Z^1(\pi_1 M, \sln_\rho)$ are integrable. 
In what follows we will prove that all obstructions vanish, by using the fact that the obstructions vanish on the boundary.
 Let 
$u_1,\ldots,u_k\co\pi_1M\to \sln$ be given such that 
\[\rho_k(\gamma)=\exp\left(\sum_{i=1}^kt^iu_i(\gamma)\right)\rho(\gamma)\]
 is a homomorphism modulo $t^{k+1}$. Then the restriction
 $i^*\rho_k\co\pi_1(\partial M)\to \SLn[n,\C\llbracket t\rrbracket]$ is also a formal deformation of order $k$.
Since $i^*\rho$ is a smooth point of the representation variety $R_n(\Z\oplus\Z)$, the formal implicit function theorem gives that $i^*\rho_k$ extends to a formal deformation of order $k+1$ (see \cite[Lemma~3.7]{Heusener-Porti-Suarez2001}). Therefore, we have that
\[0=\zeta_{k+1}^{(i^*u_1,\ldots,i^*u_k)}=i^*\zeta_{k+1}^{(u_1,\ldots,u_k)}\]
Now, $i^*$ is injective and the obstruction $\zeta_{k+1}^{(u_1,\ldots,u_k)}$ vanishes.

Hence all cocycles in $Z^1(\Gamma, \sln_\rho)$ are integrable. By applying Artin's theorem \cite{Artin1968} we obtain from a formal deformation of $\rho$ a convergent deformation (see \cite[Lemma~3.3]{Heusener-Porti-Suarez2001} or \cite[\S~4.2]{BenAbdelghani2000}).

Thus $\rho$ is a regular point of the representation variety $R_n(\pi_1M)$. 
Hence, $\dim H^1(\pi_1M;\sln_\rho)=n-1$ and the exactness of
\[0\to H^0(\pi_1M;\sln_\rho)\to\sln_\rho \to B^1(\pi_1M;\sln_\rho)\to 0\] implies
\[ \dim_\rho R_n(\pi_1M) = \dim Z^1(\pi_1M;\sln_\rho)
= n^2+n-2 - \dim H^0(\pi_1M;\sln_\rho)\,.\]
Finally, the proposition follows from Lemma~\ref{lem:smoothness}.
\end{proof}

\begin{prop}\label{prop:dimH1}
Let $K\subset S^3$ be a knot, $\lambda\in\C^*$ and $n\geq3$.
Suppose that $\lambda^2$ is a simple root of the Alexander polynomial $\Delta_K(t)$
and let $\rho_\lambda^z\co\Gamma_K\to \SLn[2]$ be a non-abelian representation as in 
\eqref{eq:redmetab}. 

If $\Delta_K(\lambda^{2i})\neq 0$ for $2\leq i \leq n-1$ then for
$\rho_{\lambda,n}^z:=\rep_n\circ\rho_\lambda^z\co\Gamma_K\to \SLn$ 
we have 
\[ \dim H^1(\Gamma_K ; \sln_{\rho_{\lambda,n}^z}) 
= (n-1) \text{ and } H^0(\Gamma_K ; \sln_{\rho_{\lambda,n}^z})=0\,.
\]
\end{prop}
\begin{proof}
A proof of the cohomological calculation will be given in Section~\ref{sec:cohocalculation}.
\end{proof}

\begin{proof}[Proof of Theorem~\ref{thm:smoothR_n}]
It follows directly from Propositions~\ref{prop:smoothpoint} and \ref{prop:dimH1} that $\rho_{\lambda,n}^z$ is a smooth point of $R_n(\Gamma_K)$ which is contained in a unique component $R_{\lambda,n}\subset R_n(\Gamma_K)$,
$\dim R_{\lambda,n} = n^2+n-2$.

That $\rho_{\lambda,n}^z$ is the limit of irreducible representations which are contained in the component $R_{\lambda,n}$ follows from Proposition~\ref{prop:irred-exist}.
\end{proof}

\section{Cohomological calculations}
\label{sec:cohocalculation}

 For the convenience of the reader we recall some facts from the representation theory of $\SLn[2]$.
 The general reference for this topic is Springer's LNM \cite{Springer1977}.

\subsection{Representation theory of $\SLn[2]$}\label{sec:sl2} 
Let $V$ be an $n$-dimensional complex vector space. In what follows we will call a homomorphism 
$r\co\SLn[2]\to \GLn[V]$ an \emph{$n$-dimensional representation} of $\SLn[2]$. The vector space $V$ turns into an $\SLn[2]$-module. Two $n$-dimensional representations 
$r\co\SLn[2]\to \GLn[V]$ and
$r'\co\SLn[2]\to \GLn[V']$ are called \emph{equivalent} if there is an isomorphism 
$\phi \co V\to V'$ which commutes with the action of $\SLn[2]$ i.e.\
$r'(A)\, \phi = \phi\, r(A)$ for all $A\in\SLn[2]$. It is clear that equivalent representations give rise to isomorphic $\SLn[2]$-modules.

 We let $\SLn[2]$ act as a group of automorphisms on the polynomial algebra $R=\C[X,Y]$.
 If $\big(\begin{smallmatrix} a & b \\c & d\end{smallmatrix}\big)\in\SLn[2]$ then there is a unique automorphism $\rep\big(\begin{smallmatrix} a & b \\c & d\end{smallmatrix}\big)$ of $R$ given by
 \[ 
 \rep\big(\begin{smallmatrix} a & b \\c & d\end{smallmatrix}\big) (X) = d\,X - b\,Y
 \quad\text{ and }\quad
 \rep\big(\begin{smallmatrix} a & b \\c & d\end{smallmatrix}\big) (Y) = -c\,X +a\,Y\,.
 \]
 We let $R_{n-1}\subset R$ denote the $n$-dimensional subspace of homogeneous polynomials of degree $n-1$. The monomials $e_l^{(n-1)} = X^{l-1}Y^{n-l}$, $1\leq l\leq n$, form a basis of $R_{n-1}$ and 
 $\rep\big(\begin{smallmatrix} a & b \\c & d\end{smallmatrix}\big)$ leaves $R_{n-1}$ invariant.
  In what follows we will identify $R_{n-1}$ and $\C^n$ by fixing the basis 
 $(e_1^{(n-1)},\ldots,e_n^{(n-1)})$ of $R_{n-1}$.
We obtain an $n$-dimensional  representation 
 $\rep_n\co\SLn[2] \to \GLn[R_{n-1}]\cong \GLn $.

 The representation $r_n$ is \emph{rational} i.e.\ the coefficients of the matrix coordinates of 
 $\rep_n\big(A)$ are polynomials in the matrix coordinates of $A$. We will make use of the following theorem.
 \begin{theorem} \label{thm:sl2-rep}
 \begin{enumerate}
 \item\label{thm:sl2-rep.1} The representation $\rep_n$ is \emph{irreducible} i.e.\ there is no $\SLn[2]$-stable invariant subspace $V$, $\{0\}\subsetneq V\subsetneq R_{n-1}$ and any irreducible rational representation of $\SLn[2]$ is equivalent to some $\rep_n$.
  \item\label{thm:sl2-rep.2} For an arbitrary rational representation $r\co\SLn[2]\to\GLn[V]$ the
  $\SLn[2]$-module $V$ is isomorphic to a direct sum of $R_n$, 
  \[ V\cong \bigoplus_{d\geq 0} R_d^{m(k)}\,.\]
 \end{enumerate}
 \end{theorem}
 \begin{proof}
See Lemma~3.1.3 and Proposition~3.2.1 of \cite{Springer1977}.
 \end{proof}
 
 It is easy to see, and it follows also from the general theory, that $\rep_n$ maps an unipotent matrix
$ \big(\begin{smallmatrix} 1 & b \\0 & 1\end{smallmatrix}\big)$ and
$ \big(\begin{smallmatrix} 1 & 0 \\c & 1\end{smallmatrix}\big)$ onto an unipotent element of $\SLn[R_{n-1}]$.
Moreover, an explicit calculation shows that the image of a diagonal matrix is
the diagonal matrix $\rep_n\big(\mathrm{diag}(a,a^{-1})\big) = \mathrm{diag}(a^{n-1},a^{n-3},\ldots,a^{-n+3},a^{-n+1})$. 
Hence the image of $\rep_n$ is contained in $\SLn[R_{n-1}]\cong\SLn$.
\begin{example}
\label{ex:Ad=r3} The representation $r_1\colon \SLn[2]\to\SLn[1]=\{1\}$ is the trivial representation. 
The representation $\rep_2\co\SLn[2]\to\SLn[2]$ is equivalent to the identity:
\[
r_2 \begin{pmatrix} a & b \\c & d\end{pmatrix} =
\begin{pmatrix} a & -b \\-c & d\end{pmatrix}=
\begin{pmatrix} i & 0 \\ 0  & -i\end{pmatrix}
\begin{pmatrix} a & b \\c & d\end{pmatrix}
\begin{pmatrix} -i & 0 \\ 0  & i\end{pmatrix}\,.
\]
Moreover,
it is easy to see that the adjoint representation $\mathrm{Ad}\co\SLn[2]\to\mathrm{Aut}(\sln[2])$ is equivalent to $\rep_3$.
\end{example}

The Lie algebra $\sln$ of $\SLn$ turns into an $\SLn[2]$-module via $\mathrm{Ad}\circ\rep_n$ where
$\mathrm{Ad}\co\SLn\to\mathrm{Aut}(\sln)$ denotes the adjoint representation. For this action we have the classical  
formula of Clebsch--Gordan:
\begin{equation}\label{eq:Clebsch-Gordan}
\mathrm{Ad}\circ\rep_n \cong \bigoplus_{i=1}^{n-1} \rep_{2i+1} \,.
\end{equation}

Let $B_n\subset\SLn $ denote the Borel subgroup of upper triangular matrices.
The vector space $R_{n-1}$ turns into a $B_2$-module via the restriction of $\rep_n$ to $B_2$.
An explicit calculation gives
\begin{equation}\label{eq:action}
 \rep_n  \big(\begin{smallmatrix} \lambda & \lambda^{-1} b \\0 & \lambda^{-1}\end{smallmatrix}\big).\, e_l^{(n-1)}
= \lambda^{n-2l+1} \sum_{j=0}^{l-1} (-b)^j \binom{l-1}{j} e_{l-j}^{(n-1)}\,.
\end{equation}
Hence $\rep_n(B_2)$ is contained in $B_n\subset\SLn$ and  the one-dimensional vector space
$\langle e^{(n-1)}_1 \rangle$ is $B_2$ invariant:
$\rep_n  \big(\begin{smallmatrix} \lambda & \lambda^{-1} b \\0 & \lambda^{-1}\end{smallmatrix}\big).
e^{(n-1)}_1 = \lambda^{n-1} e^{(n-1)}_1$. 
For a given integer $i\in\Z$ we let $\chi_i\co B_2\to \C^* = \GLn[1,\C]$ denote the rational character given by 
\[ \chi_i\big(\begin{smallmatrix} \lambda & \lambda^{-1} b \\0 & \lambda^{-1}\end{smallmatrix}\big) =
\lambda^i\,.\]
Now $\C$ turns into a $B_2$-module via $\chi_i$ i.e.\
$\big(\begin{smallmatrix} \lambda & \lambda^{-1} b \\0 & \lambda^{-1}\end{smallmatrix}\big).x=
\lambda^i  \, x$ for $x\in\C$. We will denote this $B_2$-module by $\C_{\chi_i}$.
It follows that the $B_2$-module $\langle e^{(n-1)}_1 \rangle\in R_{n-1}$ is isomorphic to
$\C_{\chi_{n-1}}$ and we obtain a short exact sequence of $B_2$-modules
\begin{equation}\label{eq:seq1}
 0\to  \C_{\chi_{n-1}} \to R_{n-1} \to \bar R_{n-1}\to 1
\end{equation}
where $\bar R_{n-1}$ denotes the quotient $R_{n-1}/\langle e^{(n-1)}_1 \rangle$.  
For a given element $x\in R_{n-1}$ we let $\bar x \in\bar R_{n-1}$ denote the class represented by $x$ i.e.\ $\bar x = x + \langle e^{(n-1)}_1 \rangle$.

For abbreviation, we will drop the representation $r_n$ from the notation and write
for $x\in R_{n-1}$
\[ \big(\begin{smallmatrix} \lambda & \lambda^{-1} b \\0 & \lambda^{-1}\end{smallmatrix}\big).x
\text{ instead of } \rep_n\big(\begin{smallmatrix} \lambda & \lambda^{-1} b \\0 & \lambda^{-1}\end{smallmatrix}\big).\, x\,.\]

\begin{lemma}
The linear map $\phi_{n-3}\co R_{n-3}\to \bar R_{n-1}$ defined by
\[\phi_{n-3}(e_l^{(n-3)}) = \frac 1 l \, \bar e_{l+1}^{(n-1)}, \quad l=1,\ldots,n-2,\] 
is an injective $B_2$-module morphism i.e. for all $x\in R_{n-3}$ we have
\[  \big(\begin{smallmatrix} \lambda & \lambda^{-1} b \\0 & \lambda^{-1}\end{smallmatrix}\big).
\phi_{n-3}(x) =
\phi_{n-3}\Big( \big(\begin{smallmatrix} \lambda & \lambda^{-1} b \\0 & \lambda^{-1}\end{smallmatrix}\big). x \Big).
\]
\end{lemma}
\begin{proof}
The linear map $\phi_{n-3}$ is injective since the vectors $\bar e_l^{(n-1)}$, $2\leq l\leq n$, form a basis of $\bar R_{n-1}$. Now
\begin{align*}
\big(\begin{smallmatrix} \lambda & \lambda^{-1} b \\0 & \lambda^{-1}\end{smallmatrix}\big).\, 
\phi_{n-3}(e_l^{(n-3)}) &=
\lambda^{n-2l-1} \frac{1}{l} \sum_{j=0}^{l} (-b)^j \binom{l}{j} \bar e_{l-j+1}^{(n-1)}\,.
\intertext{Since $\binom l j (l-j) = l \binom{l-1} {j}$ and $\bar e_{1}^{(n-1)} =0$ it follows}
\big(\begin{smallmatrix} \lambda & \lambda^{-1} b \\0 & \lambda^{-1}\end{smallmatrix}\big).\, \phi_{n-3}(e_l^{(n-3)}) &=
\lambda^{(n-2)-2l+1} \sum_{j=0}^{l-1} (-b)^j \binom{l-1}{j}\frac{1}{l-j} \bar e_{l-j+1}^{(n-1)}\\
&= \phi_{n-3} \Big( \big(\begin{smallmatrix} \lambda & \lambda^{-1} b \\0 & \lambda^{-1}\end{smallmatrix}\big).\, e_l^{(n-3)} \Big)\,.
\end{align*}
Hence $\phi_{n-3}$ is a $B_2$-module morphism.
\end{proof}

\begin{lemma}
There is a short exact sequence of $B_2$-modules 
\begin{equation}\label{eq:seq2}
0\to  R_{n-3}\xrightarrow{\phi_{n-3}} \bar R_{n-1} \to \C_{\chi_{-n+1}}\to 0.
\end{equation}
\end{lemma}
\begin{proof}
Again the lemma follows from Equation~\eqref{eq:action}:
\[ \big(\begin{smallmatrix} \lambda & \lambda^{-1} b \\0 & \lambda^{-1}\end{smallmatrix}\big). e^{(n-1)}_n
\equiv \lambda^{-n+1} e^{(n-1)}_n \bmod \langle e_1^{(n-1)},\ldots,e_{n-1}^{(n-1)} \rangle\,. \qedhere\]
\end{proof}

Let us fix a representation $\rho_\lambda^z\co\Gamma_K\to B_2$.
Then $R_n$ turns into a $\Gamma_K$-module and the exact sequences \eqref{eq:seq1} and \eqref{eq:seq2} are exact sequences of $\Gamma_K$-modules. Note that
$\C_{\chi_k}\cong \C_{\lambda^k}$ since for all $\gamma\in\Gamma_K$ and $k\in\Z$ the equation
$\chi_k\big(\rho_\lambda^z(\gamma)\big)=\lambda^{k\varphi(\gamma)}$ holds.

\begin{lemma}\label{lem:dim-H1Rn} Let $\lambda\in\C^*$, $\lambda\ne 1$, and $n>3$ be given.
If $\Delta_K(\lambda^{n-1})\neq 0$ and if $\lambda^{n-1}\neq 1$ then 
\[H^*(\Gamma_K; R_{n-1})\cong H^*(\Gamma_K; R_{n-3})\,.\]
\end{lemma}
\begin{proof}
The long exact cohomology sequences \cite[III.\S6]{Brown1982}
associated to the short exact sequences \eqref{eq:seq1} gives:
\[ 
H^k(\Gamma_K;\C_{\lambda^{n-1}})\to H^k(\Gamma_K;R_{n-1})\to
H^k(\Gamma_K;\bar R_{n-1}) \to H^{k+1}(\Gamma_K;\C_{\lambda^{n-1}})
\]
is exact for $k=0,1,2$. 
Now $H^0(\Gamma_K;\C_{\lambda^{n-1}})=0$ since $\lambda^{n-1}\neq 1$ and for $k=1,2$ the group
$H^k(\Gamma_K;\C_{\lambda^{n-1}})=0$  since $\Delta_K(\lambda^{n-1})\neq 0$ 
(see Lemma~\ref{lem:H1GammaC}). Hence
\[
 H^k(\Gamma_K;R_{n-1})\xrightarrow{\cong} H^k(\Gamma_K;\bar R_{n-1})
 \quad\text{ for $k=0,1,2$.}
\]

Finally, the short exact sequence~\eqref{eq:seq2}, Lemma~\ref{lem:H1GammaC} and the 
assumptions $\Delta_K(\lambda^{n-1})\neq 0$ with $\lambda^{n-1}\neq 1$ give that
\[
 H^k(\Gamma_K;R_{n-3})\xrightarrow{\cong} H^k(\Gamma_K;\bar R_{n-1}) 
 \quad\text{ for $k=0,1,2$}
\]
are isomorphisms (note that $\Delta_K(t)$ is symmetric).
\end{proof}

\begin{prop}\label{prop:reducedimH1}
Let $\lambda\in\C^*$ such that $\Delta_K(\lambda^2)=0$, $n\geq3$ and 
$\rho_\lambda^z\co\Gamma_K\to B_2$ be given as in \eqref{eq:redmetab}. 
If $\Delta_K(\lambda^{2k})\neq 0$ and $\lambda^{2k}\neq 1$ for $2\leq k \leq n-1$ 
then for
$\rho_{\lambda,n}^z:=\rep_n\circ\rho_\lambda^z\co\Gamma_K\to B_n\subset\SLn$ 
we have 
\[ \dim H^*(\Gamma_K; \sln_{\rho_{\lambda,n}^z}) 
= (n-1)\dim H^*(\Gamma_K; R_2 )\,.
\]
\end{prop}
\begin{proof}
It follows from \eqref{eq:Clebsch-Gordan} that we have an isomorphism of $\Gamma_K$-modules:
\[ \sln_{\rho_{\lambda,n}^z} \cong \bigoplus_{k=1}^{n-1} R_{2k}\,.\]
Now Lemma~\ref{lem:dim-H1Rn} implies that
$\dim H^*(\Gamma_K, R_{2k}) = \dim H^*(\Gamma_K, R_{2})$
since $\Delta_K(\lambda^{2k})\neq 0$ and $\lambda^{2k}\neq 1$ for $2\leq k\leq n-1$.
Hence the assertion of the proposition follows.
\end{proof}

\begin{proof}[Proof of Proposition~\ref{prop:dimH1}]
Let  $\lambda\in\C^*$ and $n\in \Z$, $n\geq3$.
Suppose that $\lambda^2$ is a simple root of the Alexander polynomial $\Delta_K(t)$
and let $\rho_\lambda^z\co\Gamma_K\to B_2$ be a non-abelian representation as in~\eqref{eq:redmetab}. 

In order to apply Proposition~\ref{prop:reducedimH1} we have to show that $\lambda^{2k}\neq 1$ for $2\leq k \leq n-1$.
Suppose that there exists $k\in\Z$, $2\leq k \leq n-1$, such that $\lambda^{2k}= 1$. 
Next note that  $\lambda^{-2} = \lambda^{2k-2}$ is a root of the Alexander polynomial since $\Delta_K(t)$ is symmetric.
Therefore the assumption of the proposition implies that $k=2$ i.e.\ $\lambda^4=1$ and hence $\lambda^2 = \pm 1$.
At the same time, $\pm1$ is not a root of $\Delta_K(t)$ since $\Delta_K(1)=\pm1$ and $\Delta_K(-1)$ is an odd integer.
This gives a contradiction and hence $\lambda^{2k}\neq 1$ for $2\leq k \leq n-1$.
Therefore, Proposition~\ref{prop:reducedimH1} implies that 
\[ \dim H^*(\Gamma_K; \sln_{\rho_{\lambda,n}^z}) = (n-1)\dim H^*(\Gamma_K; R_2 )\,.\]

Finally, observe that $\sln[2]_{\rho_\lambda^z}\cong R_2$ (see Example~\ref{ex:Ad=r3}) and 
$\dim H^1(\Gamma_K; R_2 )=1$ follows from \cite[Corollary~5.4]{Heusener-Porti2005} or \cite[4.4]{Heusener-Porti-Suarez2001}.
It is easy to see that $H^0(\Gamma_K; R_2 )=0$ since $\rho_\lambda^z$ is non-abelian.
\end{proof}

\section{Examples}\label{sec:examples}
Let $K\subset S^3$ be a knot and $\lambda^2$ a simple root of $\Delta_K(t)$.
Theorem~\ref{thm:smoothR_n} implies that
if $\Delta_K(\lambda^{2k})\neq 0$ for all $k\in\Z$, $k\neq\pm 1$, then for all $n\geq 2$, $n\in\Z$, the representation space 
$R_n(\Gamma_K)$ contains a component $R_{\lambda,n}$ of dimension $n^2+n-2$. 
 Moreover,
Proposition~3.8 of \cite{Newstead1978} shows that if a component contains an irreducible representation, then generic representations on that component are irreducible.

%

\begin{corol}
Let $K\subset S^3$ be a knot with the Alexander polynomial of the figure-eight knot.

Then the representation variety $R_n(\Gamma_K)$ contains an $(n^2+n-2)$-dimensional component and the irreducible representations form an Zariski-open subset of this component.
\end{corol}
\begin{proof}
The Alexander polynomial of the figure-eight knot is $\Delta(t)=t^2-3t+1$ and its roots are
$\lambda^{\pm 2} = 3/2 \pm \sqrt{5}/2$ and no power $\lambda^{\pm 2k}$, $k\neq \pm1$, is a root of 
$\Delta(t)$.
\end{proof}

The situation for the trefoil knot $3_1$ is more complicated since the roots of its Alexander polynomial, $\Delta_{3_1}(t)=t^2-t+1$, are the primitive $6$-th roots of unity $\lambda^{\pm 2} = e^{\pm i\pi/ 3}$.
Hence $R_n(\Gamma_{3_1})$ contains  an $(n^2+n-2)$-dimensional component 
$R_{\lambda,n}$, for $n\in\{2,3,4,5\}$,  since $e^{\pm i\pi/ 3}$ is a simple root of $\Delta_{3_1}(t)$ and since $\Delta_{3_1}(e^{\pm i {k\pi}/3})\neq0$, for $k\in \{2,3,4\}$. 

Let us study the case $n=6$:
the group $\Gamma_{3_1}$ is free product with amalgamation
\[
\Gamma_{3_1} = \langle S,T \mid STS=TST\rangle\cong\langle x,y \mid x^2 = y^3 \rangle\cong
\langle x\mid -\rangle *_{\langle c\mid -\rangle} \langle y\mid -\rangle
\]
where $x=STS$, $y=TS$, and $c=x^2=y^3$ generates the center of $\Gamma_{3_1} $.
Note that a meridian $\mu$ of $3_1$ is represented by the Wirtinger generator
$\mu=S=xy^{-1}$.
Let $\rho\co\Gamma_{3_1}\to\SLn[6]$ be an irreducible representation. 
It follows from 
Schur's lemma, that if $\rho$ is irreducible then the generator of the center $x^2=c=y^3$ has to be mapped into the center 
\[
C_6 := \{ \exp(2\pi \frac{k}{6} ) I_6 \mid 1 \leq k \leq 6\}\subset \SLn[6]
\]
 of $\SLn[6]$. Notice that for each element of the center $C_6$ there are only finitely many square and  cube roots up to conjugation in $\SLn[6]$.  This implies that if $R\subset R_6(\Gamma_{3_1})$ is an irreducible component of the representation variety then the conjugacy classes represented by the elements $\rho(c)$,
$\rho(x)$, $\rho(y)$ in $\SLn[6]$ do not vary with $\rho\in R$. Now let $\lambda = e^{ i\pi/ 6}$ be a primitive $12$-th root of unity. A cohomological non-trivial cocycle 
$z\in Z^1(\Gamma_{3_1}; \C_{\lambda^2})$ is given by $z(S)=0$ and $z(T)=1$.
Therefore the representation $\rho_{\lambda}^z\co\Gamma_{3_1}\to\SLn[2]$ is given by
\[
\rho_{\lambda}^z(S) = \begin{pmatrix} \lambda & 0\\ 0 &\lambda^{-1}\end{pmatrix}
\quad\text{ and }\quad
\rho_{\lambda}^z(T) = \begin{pmatrix} \lambda & \lambda^{-1}\\ 0 &\lambda^{-1}\end{pmatrix}\,.
\]
Hence
\[
\rho_{\lambda}^z(x) = \begin{pmatrix} i & \lambda^{-1} \\ 0 &-i\end{pmatrix}
, \quad
\rho_{\lambda}^z(y) = \begin{pmatrix} \lambda^2 & \lambda^{-2}\\ 0 &\lambda^{-2}\end{pmatrix}
\quad\text{ and }\quad \rho_{\lambda}^z(c) = -I_2\,. 
\]
Proposition~\ref{prop:irred-exist} implies that $\rho_{\lambda,6}^z = r_6\circ \rho_{\lambda}^z$ is a limit of irreducible representations. Computer supported calculations show that 
$\dim H^1(\Gamma_{3_1}; R_{10})=3$ and Lemma~\ref{lem:dim-H1Rn} implies that
$\dim H^1(\Gamma_{3_1}; R_{2k})=\dim H^1(\Gamma_{3_1}; R_{2})=1$ for $k\in\{2,3,4\}$.
Hence Formula~\ref{eq:Clebsch-Gordan} implies that
\[ \dim H^1(\Gamma_{3_1}; \sln[6]_{\rho_{\lambda,6}^z}) = 7 \quad\text{ i.e.\ }\quad Z^1(\Gamma_{3_1}; \sln[6]_{\rho_{\lambda,6}^z})=42\,.\]
In order to see that $\rho_{\lambda,6}^z$ is contained in a $42$-dimensional component of $R_6(\Gamma_{3_1})$ we proceed as follows:
let 
$A=\rho_{\lambda,6}^z(x)$ and $B=\rho_{\lambda,6}^z(y)$ denote the image of $x$ and $y$ respectively. Notice that the matrices $A$ and $B$ are conjugate to
$\rep_6 \big(\begin{smallmatrix} i & 0 \\ 0 & -i\end{smallmatrix}\big)$ and 
$\rep_6 \big(\begin{smallmatrix} \lambda^2 & 0 \\ 0 & \lambda^{-2}\end{smallmatrix}\big)$. Hence
\[
A \sim
\begin{pmatrix}
i & &&&\mathbf{0}& \\ &i &&&&\\ & &i&&&\\ & &&-i&&\\ & &&&-i&\\ &\mathbf{0}&&&&-i\end{pmatrix}
\ \text{ and }\ B\sim
\begin{pmatrix}
-1 & &&&\mathbf{0}& \\ &-1 &&&&\\ & &\lambda^2&&&\\ & &&\lambda^2&&\\ & &&&\lambda^{-2}&\\ &\mathbf{0}&&&&\lambda^{-2}\end{pmatrix}.
\] 

Further note that a choice of  eigenspaces $E_A(i)$, $E_A(-i)$, $E_B(-1)$, $E_B(\lambda^2)$,
$E_B(\lambda^{-2})$ such that $E_A(i)\oplus E_A(-i)\cong \C^6$ and
$E_B(-1)\oplus E_B(\lambda^2)\oplus E_B(\lambda^{-2})\cong \C^6$ determines a representation
$\rho\co\Gamma_{3_1}\to\SLn[6]$ completely. 

Let $\mathrm{Gr}(p,n)$ denote the Grassmannian which parametrizes all $p$-dimensional subspaces of $\C^n$. 
Hence the choice of two elements in $\mathrm{Gr}(3,6)$ in generic position determines $A$ and 
the choice of three elements in $\mathrm{Gr}(2,6)$ in generic position determines $B$.
The representation will be irreducible if the eigenspaces of $A$ and $B$ are  in general position and reducible if not. 

It is well known that
$\dim \mathrm{Gr}(p,n)= p(n-p)$ and hence 
\[\dim\big(\mathrm{Gr}(3,6)\times\mathrm{Gr}(3,6)\big)=18
\quad\text{ and }\quad
\dim\big(\mathrm{Gr}(2,6)\times\mathrm{Gr}(2,6)\times\mathrm{Gr}(2,6)\big)=24\,.
\]
 Therefore, we constructed a $42$-dimensional component of representations $C\subset R_6(\Gamma_{3_1})$ which contains
 $\rho_{\lambda,6}^z = r_6\circ \rho_{\lambda}^z$ and which also contains irreducible representations.
 Note that $6^2+6-2=40<42$.
 In conclusion we have:
 \begin{corol}
 The representation variety $R_6(\Gamma_{3_1})$ contains a $42$-dimensional component $C$. The generic representation of $C$ is irreducible and $\rho_{\lambda,6}^z\in C\subset R_6(\Gamma_{3_1})$ is a smooth point.
 \end{corol}
 \begin{proof}
 Computer supported calculations give that $\dim Z^1(\Gamma_{3_1}, \sln[6]_{\rho_{\lambda,6}^z}) =42$. Additionally, we  constructed a $42$-dimensional component $C$ containing $\rho_{\lambda,6}^z$. Now, the assertion follows from Lemma~\ref{lem:smoothness}.
\end{proof}

\bibliographystyle{alpha}
\bibliography{heusener-MathSlovaca-revision}

 \end{document}